\documentclass{article}
\usepackage{CJK,CJKnumb,CJKulem,times,dsfont,ifthen,mathrsfs,latexsym,amsfonts,color}
\usepackage{amsmath,amsthm,makeidx,fontenc,amssymb,bm,graphicx,psfrag,listings,curves,extarrows}
\makeindex
\newtheorem{definition}{Definition}[section]
\newtheorem{theorem}{Theorem}[section]
\newtheorem{conjecture}{Conjecture}[section]

\newtheorem{corollary}{Corollary}[section]
\newtheorem{proposition}{Proposition}[section]
\newtheorem{remark}{Remark}[section]
\numberwithin{equation}{section}

\begin{document}
\begin{CJK*}{GBK}{song}
\title{\bf   Global injectivity of differentiable maps via W-condition in $\mathbb{R}^2$~~\thanks{Supported by NSFC.\quad\quad
		E-mails: liuw16@mails.tsinghua.edu.cn\quad\quad }}

\date{}
\author{
	{\bf Wei Liu}      \\
	\footnotesize  \it  Department of Mathematical Sciences, Tsinghua University,\\
	\footnotesize \it  Beijing 100084, China}

\maketitle

\vskip0.6in

\begin{center}
\begin{minipage}{120mm}
\begin{center}{\bf Abstract}\end{center}
~~~~~~~~ In this paper, we study the intrinsic relation between the  global injectivity of differentiable local homeomorphisms $F$ and the rate that tends to zero  of $Spec(F)$ in $\mathbb{R}^2$, where $Spec(F)$ denotes  the set of all (complex) eigenvalues  of  $DF(x)$, for all $x\in \mathbb{R}^2$.  This depends on the $W$-condition deeply, which extends  the $*$-condition and $B$-condition. The $W$-condition reveals the rate that tends to zero of real eigenvalues of $DF$
can not exceed $\displaystyle O\Big(x\ln x(\ln \frac{\ln x}{\ln\ln x})^2\Big)^{-1}$ by the half-Reeb component method. This improves the theorems of Guti\'{e}rrez-Nguyen  \cite{GN07} and Rabanal \cite{RR10}. 
The $W$-condition is optimal for the half-Reeb component method in this paper setting.
 \vskip0.23in

{\it   Key  words:} $W$-condition; Half-Reeb component; Jacobian conjecture.

 \vskip0.23in

 MSC(2010): 14R15; 14E07; 14E09.
\vskip0.23in

\end{minipage}
\end{center}
\vskip0.26in
\newpage
\section{Introduction}

  ~~~~On the long-standing Jacobian conjecture, it is still open even in the case $n=2$. There are many results on it,  see for example \cite{BCW82} and \cite{E00}.

A very important step, for example in $\mathbb{R}^2$, is the following result,
due to A. Fernandes, C. Guti\'{e}rrez, and R. Rabanal:
\begin{theorem}\label{T1}(\cite{FGR04a})
Let $X=(f,g):\mathbb{R}^2\rightarrow \mathbb{R}^2$ be a differentiable map.
For some $\varepsilon >0$, 
if 
\begin{equation}\label{S1}
Spec (X)\cap [0,\varepsilon )=\varnothing,
\end{equation}
then $X$ is injective.
\end{theorem} 
 Theorem \ref{T1} is deep. If the assumptions (\ref{S1}) replaced 
 by $0\notin $  Spec $(F)$, then the conclusion is false, even
  for polynomial map $X$, as the Pinchuck's counterexample\cite{P94}.
   Pmyth and Xavier\cite{SX96} proved that there exists $n>2$ 
   and non-injective polynomial map  
   such that Spec$(X)\cap [0, +\infty)=\varnothing.$

Theorem \ref{T1} added to a long sequence of results on Markus-Yamabe conjecture\cite{MY60} and the eigenvalue conditions of some map for injectivity in dimension two. The Markus-Yamabe Conjecture has been solved
independently in 1993 by C. Gutierrez\cite{G95} and R. Fessler\cite{F95}.  It is false in dimension $n\geqslant3$ even for polynomial vector field\cite{CEG97}.
Theorem \ref{T1} also implies that the following conjecture  is true in dimension $n=2$.
\begin{conjecture}(\cite{CM98},~Conjecture 2.1)\label{C1.4}
	Let $F:\mathbb{R}^n\rightarrow \mathbb{R}^n$ be a ${C^1 }$ map.
	Suppose there exists  an $\varepsilon>0$ such that $|\lambda|\geqslant\varepsilon$~ for all the eigenvalues $\lambda$ of $F'(x)$ and all $x\in \mathbb{R}^n$.
	Then $F$ is injective.
\end{conjecture}

The essential tool to prove Theorem \ref{T1} is making use of the concept of the half-Reeb component that we recall in Definition \ref{hRc}.

C. Guti\'{e}rrez and V. Ch. Nguyen \cite{GN07} study the geometrical  behavior of differentiable maps in $\mathbb{R}^2$ and
the following $*$-condition on the real eigenvalues of $DF$ under the half-Reeb component technique.

For each $\theta \in \mathbb{R}$, we denote by $R_\theta$ the linear rotation
\begin{equation}\label{eR}
R_{\theta }=\begin{pmatrix}\cos\theta  & -\sin \theta  \\
\sin \theta & \cos\theta 
\end{pmatrix}
\end{equation}
and define the map $F_{\theta }=R_{\theta }\circ F \circ R_{-\theta }$.

\begin{definition}(\cite{GN07}, $*$-condition)	
	A differentiable $F$ satisfies the $*$-condition if
	for each $\theta \in \mathbb{R}$,
	there does not exist a sequence $\mathbb{R}^2 \ni z_k\to
	\infty $
	such that, $F_\theta(z_k)\to T\in \mathbb{R}^2$ and
	$DF_\theta(z_k)$ has a real eigenvalue $\lambda_k \to 0$. 
\end{definition}

\begin{theorem}\label{T2}(\cite{GN07})
Suppose that  $X:\mathbb{R}^2\rightarrow \mathbb{R}^2$  is a differentiable local homeomorphism. 
Then:

(i) If $X $ satisfies $*$-condition, then $X$  is injective and its image
is a convex set. 
 
(ii) $X$ is a global homeomorphism of $\mathbb{R}^2$ if and only if $X$ satisfies $*$-condition and its image $X(\mathbb{R}^2)$ is dense in $\mathbb{R}^2$. 
\end{theorem}

$*$-condition is somewhat weaker than condition (1.1), thus one can obtain the Theorem \ref{T1} from Theorem \ref{T2} (i) by a standard procedure.

 In other new case, the essential difficulty is that the eigenvalues of $DF$ which may be tending to zero implies  $F$ is globally injective. R. Rabanal\cite{RR10} extended the $*$ condition to the so called $B$-condition.
\begin{definition}(\cite{RR10}, $B$-condition)
	The differentiable map $F:\mathbb{R}^2 \to \mathbb{R}^2$ satisfies the  $B$-condition 
	if for each $\theta\in\mathbb{R}$, there does not exist a sequence  $(x_k,y_k)\in \mathbb{R}^2$ with $ x_k \to +\infty $ such that $F_\theta(x_k,y_k)\to T\in \mathbb{R}^2$  and $DF_\theta(x_k,y_k)$ has a real eigenvalue $\lambda_k$ satisfying $\lambda_k x_k\to 0.$
\end{definition}

He obtains the following theorem where Theorem \ref{T2}  holds if one replaced $*$-condition by $B$-conditon.

\begin{theorem} \label{T3}(\cite{RR10})
	Suppose that the differentiable map $F: \mathbb{R}^2 \to \mathbb{R}^2$  satisfies the $B$-condition and $\det DF(z)\neq 0, \forall z\in \mathbb{R}^2 $, then $F$ is a topological embedding.
\end{theorem}

In fact, Theorem \ref{T3} improves the main results of \cite{GN07}, (see \cite{RP02},\cite{RR05}).

In 2014, F. Braun and V. S. Jean\cite{BJ14} considered the relation between the half-Reeb component and  Palais-Smale condition for global injectivity.

Many references on other aspects of half Reeb component including in higher dimensional situations see (\cite{RP97}, \cite{LX19},\cite{LZ19}, \cite{MJ13}, \cite{GM09}).

For example, C. Guti\'{e}rrez and  C. Maquera considered  half-Reeb components for the global injectivity in dimension 3.
\begin{theorem}(\cite{GM09})
Let $Y= (f,g,h) : \mathbb{R}^3 \to \mathbb{R}^3$ be a polynomial map such that $Spec(Y) \cap [0,\varepsilon )=\emptyset$, for some $\varepsilon >0$. If $codim(SY)\geqslant 2$, then $Y$ is a bijection.
\end{theorem} 

Recently, W. Liu prove the following theorem by the Minimax method.
\begin{theorem} \label{LZ}(\cite{LZ19})
Let $F: \mathbb{R}^n \to \mathbb{R}^n$  be a $C^1$ map,  $n\geqslant 2$.
If for some $\varepsilon >0$,
\[ 0\notin Spec(F)~~~\mbox{and}~~~ Spec(F+F^T) \subseteq  (-\infty,-\varepsilon)~\mbox{or} ~(\varepsilon,+\infty),\]
 then $F$ is globally injective.
\end{theorem} 

Let us return to study approaching to zero of the eigenvalues of $DF$ by the half-Reeb component method in $\mathbb{R}^2$.

In this paper, we define the $W$-condition and obtain the following result.

For the convenience of our statement, let us set

 $\mathcal{P}=\Big\{P~\big|~\mathbb{R}^+\to \mathbb{R}^+, P ~\mbox{is nondecreasing  and} 
 ~\forall M>0, \mbox{there exists large constant} $
$N$ which depends on $M$ and $P$, such that
$\displaystyle\int_{2}^{N}\frac{1}{P(x)}dx> M\Big\}.$

Thus, $\mathcal{P}$ contains many functions, such as 1, $x$, $x\ln(x+1)$,
$x\ln(1+x)\ln\big(1+\ln(1+x)\big)$ and it does not include 
$x^\alpha, \forall \alpha>1; ~x\ln^\beta (x+1), \forall\beta>1.$

\begin{definition}($W$-condition) \label{W}
	
	A differentiable map $F$ satisfies the $W$-condition 
	if for each $\theta\in\mathbb{R}$ (see (\ref{eR})), there does not exist a sequence $(x_k,y_k)\in \mathbb{R}^2$ with $x_k \to +\infty $ such that $F_\theta(x_k,y_k)\to T\in \mathbb{R}^2$  and $DF_\theta(x_k,y_k)$ has a real eigenvalue $\lambda_k$ satisfying $\lambda_k P(x_k)\to 0$,
	where $P\in \mathcal{P} $.
\end{definition}

\begin{remark}\label{rm1}
	$W$-condition obviously contains $*$-condition and $B$-condition.
	Let $P(x)=x\ln(x+1)\in \mathcal{P}$, the $W$-condition with the $P$ is weaker than
	$*$-condition and $B$-condition. It seems can not be improved in this setting by making use of the half-Reeb component method. The $W$-condition profoundly reveals  the optimal rate that tends to zero of eigenvalues of $DF$
	must be in the interval $\displaystyle \Big(O(x\ln^\beta x)^{-1}, \forall \beta>1$, 	$O\Big(x\ln x\big(\ln \frac{\ln x}{\ln\ln x}\big)^2\Big)^{-1} \Big] $ by the half-Reeb component method.
\end{remark}

\begin{remark}\label{rm2}
     If $x_k$ exchanges $y_k$ in definition \ref{W}, then it is also applied.  
	
\end{remark}

\begin{remark}\label{rm3}
For example, let $g(x,y)$ is a $C^1$ function such that $g(x,y)=\frac{y}{x\ln x}$ where  $x\geqslant 2$. The map $F(x,y)=(e^{-x},g(x,y))$ satisfies $\det JF=-e^{-x}\frac{1}{x\ln x}\neq 0$.
Then, for $\{x_k\}\subseteq [2,+\infty)$, $F(x_k, 0)=(e^{-x_k},0)\to P=(0,0)$, as $x_k\to +\infty$. $JF(x_k,0)$ has a real eigenvalue
\[\frac{1}{x_k\ln x_k}=\lambda_k\to 0.\]
However, the limit of the product $x_k\ln x_k$ is away from zero.
\end{remark}

We use the $W$-condition and obtain the following results.

\begin{theorem} \label{main1}
Let $F: \mathbb{R}^2 \to \mathbb{R}^2$ be a differentiable local homeomorphism. If $ F $ satisfies $W$-condition, then $F$ is  injective
and $F(\mathbb{R}^2)$ is convex.
\end{theorem} 
Obviously, Theorem \ref{main1} implies Theorem \ref{T2} and Theorem \ref{T3}(i).

Next, we also have the following results.
\begin{theorem}\label{main2}
	Let $F: \mathbb{R}^2 \to \mathbb{R}^2$ be a differentiable Jacobian map. If $ F $ satisfies $W$-condition, then $F$ is a globally injective, measure-preserving map with convex image.
\end{theorem} 
The Theorem \ref{main2}  improves the main results of  GN  \cite{GN07} and Ra \cite{RR10}.

Since the map is injective in Theorem in  \ref{main2}, we obtain some fixed point theorem,
that's the following corollary.

\begin{corollary}\label{c1}
	If $F$ is as  in Theorem \ref{main2} and $Spec(F) \subseteq  \{z\in \mathbb{C} \big| |z|<1 \}$ , then $F$ has 
	at most one fixed point.
\end{corollary} 

Another important property on the Keller maps as in corollary \ref{c1} 
is theroem $B$ in \cite{CGM99}.
It proves that a global attractor for the discrete dynamical system has
a unique fixed point.

By the Inverse Function Theorem, the $F$ in  Theorem \ref{main1} is locally 
injective at any point in  $\mathbb{R}^2$. However, in general, it is not
global injective map. So the goal is to give the sufficient conditions in order
to get the global injectivity of $F$. We use the $W$ condition to get the following
theorem.

\begin{theorem}\label{main3}
	Let $F=(f,g):\mathbb{R}^2 \to \mathbb{R}^2$ be a local homeomorphism such that for some $s>0,~~F|_{\mathbb{R}^2\backslash D_s}$ is differentiable. If $ F $ satisfies the $W$-condition, then it is a globally injective and $F(\mathbb{R}^2)$ is a convex set.
\end{theorem} 

\begin{remark}\label{re4}
If the graph of $F$ is an algebraic set , then the injectivity of $F$	must
be the bijectivity of $F$.	
\end{remark}

$W$ condition can  be also devoted to study the differentiable map
$F:{\mathbb{R}^2\backslash D_s} \to \mathbb{R}^2$ whose the $Spec(F)$ is disjoint
with $[0,+ \infty )$.

\begin{theorem}\label{main4}
	Let $F=(f,g):\mathbb{R}^2 \backslash \overline{D_\sigma} \to \mathbb{R}^2$ be a  differential map which satisfies the $W$-condition. If $Spec(F)\cap[0,+ \infty )=\emptyset$ or $Spec(F)\cap(- \infty,0]=\emptyset$, then there exists $s\geqslant \sigma$ such that $F|_{\mathbb{R}^2\backslash D_s}$
	can be extended to an injective local homeomorphism $\widetilde{F}=(\widetilde{f},\widetilde{g}):\mathbb{R}^2\to \mathbb{R}^2$.
\end{theorem} 

These works are related to the Jacobian conjecture which can be reduce to that for all dimension $n\geqslant 2$,
a polynomial map $F: \mathbb{C}^n \to  \mathbb{C}^n$ of the form
$F=x+H$,~where $H$ is cube-homogeneous and $JH$ is symmetry, is injective
if $Spec (F)=\{1\}$. (see \cite{BV05}).
\vskip0.09in

 In order to prove our theorems, we need to use
 the definition and propositions of the half-Reeb component.

\section{Half-Reeb component}

In this section, we will introduce some preparation on the eigenvalue conditions of $Spec (F)$. 

Let $h_0(x,y)=xy$ and consider the set 
$$B=\{(x,y)\in [0,2]\times [0,2]\big|0<x+y\leqslant 2\}.$$

\begin{definition}(half-Reeb component\cite{G95})\label{hRc}
	Let $X$ be a  differentiable map from $\mathbb{R}^2\to \mathbb{R}^2 $.
	 $\det DX_p\neq ,\forall p\in \mathbb{R}^2$, Given  $h\in\{f,g\}$, we will say that $\mathcal{A}\subseteq \mathbb{R}^2$ is a half-Reeb component for $\mathcal{F}(h)$ \big(or simply a hRc for $\mathcal{F}(h)$\big)if there exists a homeomorphism
	 $H : B\to \mathcal{A}$ which is a topological equivalence between  $\mathcal{F}(h)|_\mathcal{A}$ and $\mathcal{F}(h_0)|_B$ and such that:
	 
	  (1) The segment$ \{(x,y)\in B:x+y=2\}$ is sent by
	  $ H $ onto a transversal section for the foliation $\mathcal{F}(h) $ in the complement of $H(1,1)$; this section is called the compact edge of $\mathcal{A}$;
	  
	  (2) Both segments $\{(x,y)\in B:x=0\}$ and $\{(x,y)\in B:y=0\}$ are sent by $H$ onto full half-trajectories of $\mathcal{F}(h)$. These two semi-trajectories of $\mathcal{F}(h)$ are called the noncompact edges of $\mathcal{A}$.
  	
\end{definition}

\begin{proposition}\cite{FGR04a}\label{2.1}
 Suppose that $X=(f,g): \mathbb{R}^2\in \mathbb{R}^2$ is a differentiable map such that $0\notin Spec(X)$. If $X$ is not injective, then both $\mathcal{F}(f)$ and $\mathcal{F}(g)$ have half-Reeb components.
\end{proposition}

  \begin{proposition}\cite{FGR04a}\label{2.2}
   Let $X=(f,g): \mathbb{R}^2\in \mathbb{R}^2$ be a non-injective, differentiable map such that $0\notin Spec(X)$: Let $\mathcal{A}$ be a hRc of $\mathcal{F}(f)$ and let $(f_\theta,g_\theta )=R_\theta \circ X \circ R_{-\theta}$, where $\theta \in \mathbb{R}$ and $R_\theta$ is in (\ref{eR}). If $ \Pi(x,y)=x $ is bounded, where $  \Pi :\mathbb{R}^2\to \mathbb{R} $ is given by $ \Pi(x,y)=x $, then there is an $ \varepsilon >0 $ such that, for all $ \theta \in (-\varepsilon,0)\cup (0, \varepsilon) $; $ \mathcal{F}(f_\theta) $ has a hRc $\mathcal{A}_\theta $ such that $ \Pi(\mathcal{A}_\theta )$is an interval of infinite length.
  \end{proposition}

\vskip0.1in
%%%%%%%%%%%%%%%%%%%%%%%%%%%%%%%%%%%%%%%%%%%%%%%%%%%%%%%%%%%%%%%%%%%%%%%%
%%%%%%%%%%%%%%%%%%%%%%%%%%%%%%%%%%%%%%%%%%%%%%%%%%%%%%%%%%%%%%%%%%%%%%%%
%%%%%%%%%%%%%%%%%%%%%%%%%%%%%%%%%%%%%%%%%%%%%%%%%%%%%%%%%%%%%%%%%%%%%%%%
\section{Half-Reeb component and $W$-condition}
 In this section, we will establish the essential fact that 
 the $W$-condition ensures non-existence of half-Reeb component.
 
 Let $F=(f,g):\mathbb{R}^2\to \mathbb{R}^2$ be a local homeomorphism of $\mathbb{R}^2$. For each $\theta\in\mathbb{R} $, we denoted by $R_\theta$
 the linear rotation $\big($see (\ref{eR}$)\big)$:
 \[ (x,y)\to (x\cos\theta-y\sin\theta,x\sin\theta+y\cos\theta),\]
 and
 \[ F_\theta:=(f_\theta,g_\theta) =R_\theta \circ F \circ R_{-\theta}.\]
 
 In other words, $F_\theta$ represents the linear rotation $R_\theta$ in the 
 linear coordinates of $\mathbb{R}^2.$ 
 
  \begin{proposition}\label{3.1}
A differentiable local homeomorphism $F: \mathbb{R}^2\to \mathbb{R}^2$ which satisfies $W$-condition has no half-Reeb components.
 \end{proposition}

\begin{proof} Suppose by contradiction that $F$ has a half-Reeb component. 
	In order to obtain this result, we consider the map $(f_\theta,g_\theta)=F_\theta$.
	From  Proposition \ref{2.2}, there exists some  $\theta\in\mathbb{R}$, such 
	that $\mathcal{F}(\mathcal{A}_\theta)$ has a half-Reeb component which $\Pi (\mathcal{A})$ is unbounded interval, where $\Pi (\mathcal{A})$ denote orthgonal projection onto the first coordinate in $\mathcal{A}$. Therefore $\exists$ $b$ and a half-Reeb component 
$A$, such that $[b,+\infty)  \subseteq \Pi (\mathcal{A}) $. Then, for large enough
$a>b$ and any $x\geqslant a$,  the vertical line $\Pi^{-1}(x)$ intersects exactly the  one trajectory
$\alpha_x \cap [x, +\infty) = x$, i.e. $x$ is maximum of the the trajectory
$\Pi_{\alpha_x}$. If $x\geqslant a$, the intersection $\alpha_x \cap \Pi^{-1}(x)$  is compact
subset in $\mathcal{A}$. 

Thus, we can define functions
$H: (a, +\infty) \to \mathbb{R}$ by 
\[ H(x)=\sup \{y: (x,y) \in \Pi^{-1}(x)\cap \alpha_x\}.\]

As $\mathcal{F}(f_\theta)$ is a foliation and can obtain

\[\Phi: (a, +\infty) ~~\mbox{ by}~~\Phi(x)=f_\theta \big(x,H(x)\big).\]

We can know that $\Phi$ is a bounded, strictly monotone function such that, for some full measure subset $M \subseteq  (a,+\infty)$.

Since the image of $\Phi$ is contained in $f_\theta(\Gamma)$ where $\Gamma$
is compact edge of hRc $\mathcal{A}$, the function $\Phi$ is bounded in $(a,+\infty)$. Furthermore, $\Phi$ is continuous  because $\mathcal{F}(f_\theta)$ is a $C^0$ foliation. And
since $\mathcal{F}(f_\theta)$ is transversal to $\Gamma$, $\Phi$ is monotone strictly. 

For the measure subset $M\subseteq (a,+\infty) $, such that $\Phi(x)$ is 
differentiable on $M$ and the Jacobian matrix of $F_\theta(x,y)$  at
$\big(x, H(x)\big)$ is 

\[ DF_\theta(x,H(x))=
\begin{pmatrix}
\Phi'(x)  & 0 \\
\partial_xg_\theta\big(x,H(x)\big) & \partial_yg_\theta\big(x,H(x)\big)\end{pmatrix}. \]

Therefore, $\forall x\in M$, 
$\Phi'(x)=\partial_x f_\theta \big(x,H(x)\big)$ is a real eigenvalue of $DF_\theta(x,H(x))$ and we denote it by $\lambda(x):=\Phi'(x)$.

Since $F$ is  local homeomorphism,
without loss of generality, we assume $\Phi$ is strictly monotone increasing, $\Phi'(x)>0, \forall x\in M$.  Let any  function $P\in \mathcal{P}$, where

 $\mathcal{P}=\Big\{P~\big|~\mathbb{R}^+\to \mathbb{R}^+, P ~\mbox{is nondecreasing  and} 
~\forall M>0, \mbox{there exists large constant } $
$N$ which depends on $M$ and $P$, such that
$\displaystyle\int_{2}^{N}\frac{1}{P(x)}dx> M\Big\}.$

Claim:  \[\liminf_{x_k\rightarrow +\infty} \Phi'(x_k)P(x_k)>0.\]

Because $P(x)$ and $\Phi'(x)$  are  both positive, we can 
suppose by contradition that

\noindent $\liminf_{x_k\rightarrow +\infty} \Phi'(x_k)P(x_k)=0.$ 
There exists a subsequence denoted still $\{ x_k\} $ with $x_k \to +\infty$ such that $ \Phi'(x_k)P(x_k) \to 0.$
That is  $ \lambda(x_k)P(x_k) \to 0.$ Since $F_\theta(\mathcal{A})$ is bounded, $F_\theta\big(x_k, H(x_k)\big)$ converges to a finite value $T$ on compact set $\overline{\mathcal{F}_\theta (\mathcal{A})} $.
This contradicts  the $W$-condition.

Therefore, there exist constant  $a_0~(a_0>2) $
and small $ \varepsilon_0 >0 $, such that 
\[\Phi'(x)P(x)>\varepsilon_0,~\forall x \geqslant a_0. \]

Since $ \Phi(x) $ is bounded, there exists $ L>0 $, such that
\[ \Phi(x)-\Phi(a_0)\leqslant L, ~~\forall x\geqslant a_0.\]

 By the definiton of  $\mathcal{P}$, we can choose $C$ large enough, such that 
\[ \int_{a_0}^{C}\frac{1}{P(x)}dx >\frac{L}{\varepsilon_0}. \]

Thus,
\[L\geqslant \Phi(C)-\Phi(a_0)=\int_{a_0}^{C}\Phi'(x)dx\geqslant \int_{a_0}^{C}\frac{\varepsilon_0}
{P(x)}dx>L. \]

It is contradiction.

\end{proof}

\section{The Proof of Theorem \ref{main1}}
\begin{proof}
Suppose by contradiction that $F$ is not injective. By Proposition \ref{2.1}, $F$ has a half-Reeb component, this 
contradicts Proposition \ref{3.1} that $F$ has no half-Reeb component
if $F$ satisfies the $W$-condition. 
\end{proof}

\section{The Proof of Theorem \ref{main2}}
\begin{proof}
Firstly, we prove the equivalence of the differential Jacobian map and measure-preserving in any dimension $n$.
 
 For any nonempty measurable set $\Omega\subset \mathbb{R}^n.$
Since $F:\mathbb{R}^n \to \mathbb{R}^n,$\\
denote~$V = \{ F(x)\left| {x \in } \right.\Omega \}$. 
Let the components of $F(x)$ be $v_i (i=1,2...n)$, i.e.
$ F({x_1},...{x_n}) = ({v_1}({x_1,...x_n}),...{v_n}({x_1...x_n})).$
So $dv = \det F'(x)dx.$
Since $\det F'(x) \equiv 1$, we get
$dv=dx.$

Therefore, $\int_V {dv}  = \int_\Omega  {dx}.$ It implies $F$ preserves measure.

Inversely, let $v = F(x),~\forall x \in \Omega.$
We still denote
$V = \{ F(x)\left| {x \in } \right.\Omega \}.$

Since $F$ preserves measure, one gets
$\int_V {dv}  = \int_\Omega  {dx}.$

Combining with
$dv = \det F'(x)dx,$
we obtain
$\int_V {dv}  = \int_\Omega  {\det F'(x)dx}.$

Thus, we have
$\int_\Omega  {dx}  = \int_\Omega  \det F'(x)dx.$
That is
\[\displaystyle \int_\Omega  {\big(1 - \det F'(x)\big)dx}  = 0,~\forall ~\Omega  \subset {\mathbb{R}^n}.\]
Claim: $\det F'(x) \equiv {\rm{1}},~\forall x \in {\mathbb{R}^n}.$
It's proof by contradiction. Suppose $\exists ~{x_0} \in \mathbb{R}^n,\det F'(x_0) \ne {\rm{1}}.$
Without loss of generality, we suppose
$\det F'({x_0})>1,$ denote $C=\det F'({x_0}) - 1>0 $.
Since $F \in {C^1}$,  $\det F'(x)\in C$.
$\exists~ \delta>0$, such that
$\det F'(x) - 1 \geqslant  \frac{C}{2},$
$\forall x \in U\left( {{x_0},\delta } \right)$.

Choosing $\Omega  = U\left( {{x_0},\delta } \right)$, thus
\[\int_{U\left( {{x_0},\delta } \right)} {(1 - \det F'(x))dx}  \le \int_{U\left( {{x_0},\delta } \right)} { - \frac{C}{2}dx}  =  - \frac{C}{2}m(U\left( {{x_0},\delta } \right))<0,\]
it contradicts.

	Thus, we obtain the global injectiveity of $F$ by the Theorem \ref{main1}.
	Forthermore, the image of $F$ is convex.
\end{proof}

  Before we give the proof of Theorem \ref{main3}, we need the following
  proposition.
   \begin{proposition}\label{3.2}
   	Let $F=(f,g):\mathbb{R}^2 \to \mathbb{R}^2$ be a local homeomorphism such that for some $s>0,~~F|_{\mathbb{R}^2\backslash D_s}$ . If $F$ satisfies the $W$ condition,
   	then\\
   	(1) any half Reeb component of $\mathcal{F}(f)$ or $\mathcal{F}(g)$ is a bounded
   	in $\mathbb{R}^2$;\\
   	(2) If $F$ extends to a local homemorphism $\widetilde {\overline F } = \left( {\overline f ,\overline g } \right):\mathbb{R}^2\to \mathbb{R}^2$ , $\mathcal{F}(\overline f)$ and $\mathcal{F}(\overline g)$ have no half-Reeb components.
   	  \end{proposition}
  \begin{proof}
  Without loss of generality, we consider the  $\mathcal{F}(f)$, by contradiction  have an unbounded half Reed component. By the process in Proposition \ref{3.1}, we assume 
  that $\mathcal{F}(f)$ has a half Reeb component $\mathcal{A}$  such that $\Pi (\mathcal{A})$ is unbounded interval. Furthermore,
  \[ DF(x,H(x))=
  \begin{pmatrix}
  \Phi'(x)  & 0 \\
  \partial_xg\big(x,H(x)\big) & \partial_yg\big(x,H(x)\big)\end{pmatrix}. \]
  
  If \noindent $\liminf_{x_k\rightarrow +\infty} \Phi'(x_k)P(x_k)=0.$ 
  There exists a subsequence denoted still $\{ x_k\} $ with $x_k \to +\infty$ such that $ \Phi'(x_k)P(x_k) \to 0.$
  That is  $ \lambda(x_k)P(x_k) \to 0.$ Since $F(\mathcal{A})$ is bounded, $F\big(x_k, H(x_k)\big)$ converges to a finite value $T$ on compact set $\overline{\mathcal{F}(\mathcal{A})} $.
  This contradicts  the $W$-condition.

If \noindent $\liminf_{x_k\rightarrow +\infty} \Phi'(x_k)P(x_k)\ne 0,$ then 
$\liminf_{x_k\rightarrow +\infty} \Phi'(x_k)P(x_k)> 0$. Thus, there exists
$C_0>C$ and $l>0$ such that $\Phi'(x)P(x)>l, \forall x>C_0$.
There exists $K>0$ such that, take $C>C_0$ \[\int_{C_0}^{C}\frac{l}{P(x)}dx >K. \]
And $\Phi(C)-\Phi(C_0)<K.$

Then
\[K<\int_{C_0}^{C}\frac{l}{P(x)}dx \le \int_{C_0}^{C}\Phi'(x)dx <K.\]
This contradiction proves the proposition.
\end{proof}

\section{The Proof of Theorem \ref{main3}}
\begin{proof}
	By Proposition \ref{3.2}, it's very easy to know the image of $F$ is convex.
	This implies that  $\mathcal{F}(f)$ has a half Reeb component.
	It contradicts the Proposition \ref{3.1}.
	Thus, we complete the proof.	
\end{proof}

\section{The Proof of Theorem \ref{main4}}
\begin{proof}
By similar methods, we can prove the Theorem \ref{main4}  by  half Reeb component
and Proposition \ref{3.2}.
\end{proof}

In finally, we prove the Corollary  $\ref{c1}$. 

${\bf The~~ Proof ~~of ~~Corollary~~ \ref{c1}}$: Consider
$G:\mathbb{R}^2\to \mathbb{R}^2$ and $G(z)=F(z)-1, \forall z\in  \mathbb{R}^2 $ .
$G(z)$ has no positive eigenvalue because $Spec(G) \subset \{ z \in {\mathbb{R}^2}:{\mathop{\rm Re}\nolimits} (z) < 0\} $. So $G$ is injective
by Theorem \ref{main1}. Thus, $F$ has a fixed point.

\begin{remark}\label{re5}
It is very important and meaningful to study the relation between half-Reeb component in higher dimensions and the rate of tending to zero of eigenvalues of $DF$. 
\end{remark}

%\bibliographystyle{abbrv}
%\bibliography{ref_general-English,ref_general-Chinese,ref_general-Internet,ref_general-Myself}
%%%%%%%%%%%%%%%%%%%%%%%%%%%%%%%%%%%%%%%%%%%%%%%%%%%%%%%%%%%%%%%%%%%%%%%%%%%%%%%%%%%%%%%%%%%%%%%%%%%%%%%%%%%%%%%%%%%%%%%%%%%
\end{CJK*}
 \end{document}